\documentclass[reqno]{amsart}
\usepackage{etex}
\usepackage[a4paper,hmarginratio=1:1]{geometry}
\usepackage{amssymb,amsfonts,amsmath}
\usepackage{comment} 
\usepackage[all,arc]{xy}
\usepackage{enumerate}
\usepackage{mathrsfs,mathtools}
\usepackage{todonotes,booktabs}
\usepackage{stmaryrd}
\usepackage{marvosym}
\usepackage{graphicx}
\usepackage{pdflscape} 	

\usepackage{bbm}
\usepackage{tikz}
\usepackage{tikz-cd}
\usetikzlibrary{trees}
\usetikzlibrary[shapes]
\usetikzlibrary[arrows]
\usetikzlibrary{patterns}
\usetikzlibrary{fadings}
\usetikzlibrary{backgrounds}
\usetikzlibrary{decorations.pathreplacing}
\usetikzlibrary{decorations.pathmorphing}
\usetikzlibrary{positioning}
	
\def\biblio{\bibliography{duality}\bibliographystyle{alpha}}

\usepackage{xcolor} 
\usepackage{graphicx}

\usepackage[pagebackref]{hyperref}

\definecolor{dark-red}{rgb}{0.5,0.15,0.15}
\definecolor{dark-blue}{rgb}{0.15,0.15,0.6}
\definecolor{dark-green}{rgb}{0.15,0.6,0.15}
\hypersetup{
    colorlinks, linkcolor=dark-red,
    citecolor=dark-blue, urlcolor=dark-green
}

\renewcommand*{\backref}[1]{}
\renewcommand*{\backrefalt}[4]{%
  \ifcase #1 %
No citations.
  \or
(cit. on p. #2).%
  \else
(cit on pp. #2).%
  \fi%
}

\usepackage[nameinlink,capitalise,noabbrev]{cleveref}

\newtheorem{thm}{Theorem}[section]
\newtheorem{cor}[thm]{Corollary}
\newtheorem{prop}[thm]{Proposition}
\newtheorem{lem}[thm]{Lemma}
\newtheorem{conj}[thm]{Conjecture}

\theoremstyle{definition}
\newtheorem{defn}[thm]{Definition}

\theoremstyle{remark}

\theoremstyle{theorem}
\newtheorem*{thm*}{Theorem}

\makeatletter
\let\c@equation\c@thm
\makeatother
\numberwithin{equation}{section}

\DeclareMathOperator{\Hom}{Hom}

\DeclareMathOperator{\End}{End}

\DeclareMathOperator{\cV}{\mathcal{V}}
\DeclareMathOperator{\id}{\mathrm{id}}

\DeclareMathOperator{\Spec}{Spec}
\DeclareMathOperator{\Mod}{Mod}
\DeclareMathOperator{\Stable}{Stable}

\DeclareMathOperator{\StMod}{StMod}

\DeclareMathOperator{\fib}{fib}

\newcommand{\Q}{\mathbb{Q}}

\newcommand{\Z}{\mathbb{Z}}

\Crefname{figure}{Figure}{Figures}
\Crefname{assu}{Assumption}{Assumptions}
\Crefname{lem}{Lemma}{Lemmas}
\Crefname{thm}{Theorem}{Theorems}
\Crefname{prop}{Proposition}{Propositions}

\newcommand{\fp}{\mathfrak{p}}
\newcommand{\fq}{\mathfrak{q}}

\newcommand{\recollement}[5]{
\xymatrix{{#1} \ar[r]|-{#2} & #3 \ar[r]|-{#4} \ar@<1ex>[l]^-{{#2}_!} \ar@<-1ex>[l]_-{{#2}^*} & #5, \ar@<1ex>[l]^-{{#4}!} \ar@<-1ex>[l]_-{{#4}^*}
}}
\let\lim\relax

\DeclareMathOperator{\lim}{lim}
\DeclareMathOperator{\Proj}{Proj}

\newcommand{\F}{\mathbb{F}}

\newcommand{\sm}{\wedge}
\DeclareMathOperator{\Tel}{Tel}

\title[Chromatic conjectures]{A short introduction to the telescope and chromatic splitting conjectures}

\author{Tobias Barthel}
\address{Department of Mathematical Sciences, University of Copenhagen, Universitetsparken 5, 2100 K{\o}benhavn {\O}, Denmark}
\email{tbarthel@math.ku.dk}
\thanks{This work is supported by the Danish National Research Foundation Grant DNRF92 and the European Unions Horizon 2020 research and innovation programme under the Marie Sklodowska-Curie grant agreement No. 751794. I would like to thank the Isaac Newton Institute (funded by EPSRC Grant number EP/R014604/1) for its hospitality during the Homotopy Harnessing Higher Structures Programme.}

\date{\today}

\begin{document}

\begin{abstract}
In this note, we give a brief overview of the telescope conjecture and the chromatic splitting conjecture in stable homotopy theory. In particular, we provide a proof of the folklore result that Ravenel's telescope conjecture for all heights combined is equivalent to the generalized telescope conjecture for the stable homotopy category, and explain some similarities with modular representation theory.
\end{abstract}

\maketitle

\setcounter{tocdepth}{1}
\tableofcontents
\def\biblio{}

This document contains a slightly expanded and updated version of an overview talk, delivered at the Talbot Workshop 2013 on chromatic homotopy theory, on two of the major open conjectures in stable homotopy theory: the telescope conjecture and the chromatic splitting conjecture. As such, these notes are entirely expositional and are not aimed to give a comprehensive account; rather, we hope some might see them as an invitation to the subject. 

We have augmented the original content of the talk by some material which is well-known to the experts but difficult to trace in the literature. In particular, we prove the folklore result that the telescope conjecture for all heights combined is equivalent to the classification of smashing Bousfield localizations of the stable homotopy category. In the final section, we discuss algebraic incarnations of chromatic structures in modular representation theory.  

We will assume some familiarity with basic notions from stable homotopy theory, and refer the interested reader to \cite{orangebook} as well as \cite{bb_handbook} for a more thorough discussion of chromatic homotopy theory. 

\subsection*{Acknowledgements}

I would like to thank Mike Hopkins and the participants of Talbot 2013 for several useful discussions on this topic as well as Malte Leip, Doug Ravenel, Gabriel Valenzuela, and the referee for comments on an earlier draft of this document. Furthermore, I am grateful to Haynes Miller and Norihiko Minami for encouraging me to revise my original notes. 

\section{Motivation: Freyd's generating hypothesis}

In 1966, Freyd~\cite{freyd_stable} proposed one of the most fundamental conjectures in stable homotopy theory:

\begin{conj}[Generating hypothesis] 
Let $f\colon X \to Y$ be a map of finite spectra with $\pi_*f = 0$, then $f$ is nullhomotopic.
\end{conj}

As of today, this hypothesis is completely open---since the computation of stable homotopy groups of finite complexes is notoriously difficult, there is essentially no evidence supporting either conclusion. However, one important statement that would follow if the hypothesis was true is that the map
\[
\pi_*\colon [X,Y] _* \longrightarrow \Hom_{\pi_*S^0}(\pi_*X,\pi_*Y)
\]
is an isomorphism for all finite spectra $X$ and $Y$, the target being the group of graded homomorphisms of $\pi_*S^0$-modules. The generating hypothesis also has a number of other curious consequences, see for example \cite{hovey_genhyp}.

In the early 1990s, Devinatz and Hopkins described a chromatic approach to the generating hypothesis in the special case when $Y=S^0$ is the sphere spectrum~\cite{devinatz_genhyp1}. We explain their idea first in the global setting. Suppose $f\colon X \to S^0$ is not null and write $f^{\vee}\colon S^0 \to DX$ for its Spanier--Whitehead dual; we have to show that $\pi_*f \neq 0$. If $f$ is of infinite order, then the question reduces to a rational statement, so suppose $f$ has finite order $d$. Recall that by Brown representability there exists a spectrum $I$ with 
\[
[W,I] \cong \Hom_{\Z}(\pi_0W,\Q/\Z)
\]
for all $W$, the so-called Brown--Comenetz dual of the sphere spectrum. 

This can be used to reduce the generating hypothesis with target $S^0$ to a set of universal examples, a strategy reminiscent of the proof of the nilpotence theorem \cite{nilpotence1}. Indeed, there is a map $f_d\colon [X,S^0] \to \Q/\Z$ sending $f$ to $1/d$. By construction of $I$, $f_d$ corresponds to a map $f_d\colon DX \to I$. Writing $I$ as a directed colimit of finite spectra $I^{\alpha}$, we see that $f_d$ factors through some $f_d^{\alpha}\colon DX \to I^{\alpha}$, i.e., there is a commutative diagram
\[
\xymatrix{S^0 \ar[r]^-{f^{\vee}} \ar[dr]_{1/d} & DX \ar[d]^{f_d} \ar@{-->}[rd]^-{f_d^{\alpha}} \\
& I & I^{\alpha}. \ar[l]}
\]
Spanier--Whitehead duality gives a map $(f_d^{\alpha})^{\vee}\colon DI^{\alpha} \to X$ such that the composite
\[
\xymatrix{DI^{\alpha} \ar[r]^-{(f_d^{\alpha})^{\vee}} & X \ar[r]^-{f} & S^0}
\]
is not nullhomotopic and depends only on $\alpha$ and $d$. Therefore, it suffices to prove the claim for these universal examples $DI^{\alpha} \to S^0$. 

In order to deal with them, we need to construct suitable models for the $I^{\alpha}$ and then prove the generating hypothesis for these examples. Instead of running this programme for $S^0$ directly, Devinatz and Hopkins propose to use the chromatic convergence theorem \cite{orangebook,chromaticcompletion}, which says that, $p$-locally, $S^0$ is equivalent to the limit of the chromatic tower
\begin{equation}\label{eq:chromatictower}
\xymatrix{\ldots \ar[r] & L_nS^0 \ar[r] & L_{n-1}S^0 \ar[r] &  \ldots \ar[r] & L_1S^0 \ar[r] & L_{0}S^0 \simeq S_{\Q}^0,}
\end{equation}
where $L_n$ denotes $E(n)$-localization (reviewed below). It consequently suffices to prove an analogue of the generating hypothesis for the $E(n)$-local analogues of the universal examples considered above, for each height $n\ge 0$. The filtration steps of the chromatic tower are built out of the monochromatic layers $M_nS^0 = \fib(L_nS^0 \to L_{n-1}S^0)$, which leads to the study of $IM_nS^0$ via Gross--Hopkins duality \cite{grosshopkins1} and the $K(n)$-local Picard group \cite{hms_pic}. The original approach relied on the telescope conjecture as well as the chromatic splitting conjecture in order to control the universal examples, and it has been carried out successfully at height 1~\cite{devinatz_genhyp1}:

\begin{thm}[Devinatz]
If $p>2$ and $f\colon X \to S^0$ a map between $p$-local finite spectra with $\pi_*f = 0$, then $L_1f$ is nullhomotopic.
\end{thm}

In response to subsequent progress on the telescope conjecture and the chromatic splitting conjecture as outlined in the next sections, Devinatz describes a modified approach in \cite{devinatz_genhyp2}, which appears to be the current state of the art.

\section{Recollections on Bousfield localization}

Throughout this section, we will implicitly work locally at a fixed prime $p$. Let $E$ be a spectrum. A spectrum $X$ is called $E$-acyclic if $E \sm X \simeq 0$ and $X$ is called $E$-local if any map from an $E$-acyclic spectrum into $X$ is nullhomotopic. Moreover, a map $f\colon X \to Y$ is called an $E$-equivalence if $E \sm f$ is an equivalence or, equivalently, if the fiber of $f$ is $E$-acyclic. A localization functor is an endofunctor $L$ of the stable homotopy category together with a natural transformation $\eta\colon \id \to L$ such that $L\eta\colon L \to L^2$ is an equivalence and $L\eta \simeq \eta L$. Based on ideas of Adams, Bousfield \cite{bousfield_locspectra} rigorously constructed a localization functor which forces the $E$-equivalences to be invertible; more precisely:

\begin{thm}[Bousfield, 1979]
If $E$ is a spectrum, there is a localization functor $L_E$ on the stable homotopy category together with a natural transformation $\eta_E\colon \id \to L_E$ such that, for any spectrum $X$, the map $\eta_E(X)\colon X \to L_EX$ exhibits $L_EX$ as the initial $E$-local spectrum with a map from $X$. The functor $L_E$ is called Bousfield localization at $E$ and the fiber $C_E$ of $\eta_E$ is called $E$-acyclization. 
\end{thm}

It follows formally that $\eta_E(X)$ is also the terminal $E$-equivalence out of $X$. The proof of this theorem relies on verifying the existence of a set of suitable generators for the category of $E$-acyclics. It is an open problem \cite[Conj.~9.1]{hovey_structure_1999} whether every localization functor on the stable homotopy category arises as localization with respect to some spectrum $E$. 

The fiber sequence $C_E \to \id \to L_E$ can be thought of as providing a way to decompose the stable homotopy category into two subcategories in a well-behaved way. We might therefore ask for a classification of all Bousfield localizations. The first result in this direction was proven by Ohkawa~\cite{ohkawa}. To state it, recall that two spectra $E$ and $F$ are said to be Bousfield equivalent if they have identical categories of acyclics, i.e., $\ker(L_E) = \ker(L_F)$. The corresponding equivalence class of $E$ is denoted by $\langle E \rangle$, so we have $\langle E \rangle = \langle F \rangle$ if and only if $L_E \simeq L_F$. As usual, we define $\langle E \rangle \vee  \langle F \rangle = \langle E \vee F \rangle$ and $\langle E \rangle \sm  \langle F \rangle = \langle E \sm F \rangle$.

\begin{thm}[Ohkawa]
The collection of Bousfield classes of spectra forms a set of cardinality at least $2^{\aleph_0}$ and at most $2^{2^{\aleph_0}}$.
\end{thm}

In light of this result, a classification of all Bousfield localizations does not seem to be feasible, see \cite{hovey_structure_1999} for some partial results. Instead, we will single out two particularly well-behaved families among all Bousfield localizations:

\begin{defn}
A localization functor $L$ is called smashing if it commutes with set-indexed direct sums or, equivalently, if the natural transformation $X \sm LS^0 \to LX$ is an equivalence for all spectra $X$. Moreover, $L$ is finite if there exists a collection of finite spectra that generates the category $\ker(L)$ of $L$-acyclics.
\end{defn}

Miller \cite{miller_finiteloc} proves that any finite localization is smashing and any smashing localization functor $L$ is equivalent to Bousfield localization at $LS^0$ by \cite{ravenel_localization}, so from now on all localization functors we consider are assumed to be Bousfield localizations.

\section{The telescope conjecture}

We start with some examples of finite and smashing localizations; as before, everything is implicitly localized at a prime $p$. Let $K(n)$ and $E(n)$ be the $n$th Morava $K$-theory and $n$th Johnson--Wilson theory, respectively, with coefficients
\[
K(n)_* = \F_{p}[v_{n}^{\pm 1}] \qquad \text{and} \qquad E(n)_*=\Z_{(p)}[v_1,\ldots,v_{n-1},v_{n}][v_n^{-1}],
\]
where $v_i$ is of degree $2(p^i-1)$. By \cite{ravenel_localization}, if a finite spectrum $F$ is $K(n)$-acyclic, then it is also $K(n-1)$-acyclic\footnote{In fact, as long as $n>1$, this result has been extended to all suspension spectra by Bousfield \cite{bousfield_kneqspaces}. For $n=1$, a counterexample is given by $K(\Z,3)$.}; since $\langle E(n)\rangle = \langle \bigvee_{i=0}^nK(i)\rangle$, this spectrum $F$ is then also $E(n)$-acyclic. A finite spectrum $F$ is of type $n$ if $n$ is minimal with the property that $K(n)_*(F) \neq 0$, and such a finite number $n$ exists for any nontrivial finite spectrum. By the periodicity theorem \cite{nilpotence2}, any finite type $n$ spectrum $F$ admits an (essentially unique) $v_n$-self map, and we write $\Tel(F) = F[v_n^{-1}]$ for the associated telescope. It then follows from the thick subcategory theorem \cite{nilpotence2} that the Bousfield class of $\Tel(F)$ depends only on $n$, so we will also write $\Tel(n)$ for $\Tel(F)$. 

\begin{defn}
Let $n\ge 0$, then we define two localization functors on the stable homotopy category by
\[
L_n^f = L_{\Tel(0) \vee \Tel(1) \vee \ldots \vee \Tel(n)} \qquad \text{and} \qquad L_n = L_{E(n)} \simeq L_{K(0) \vee K(1) \vee \ldots \vee K(n),}
\]
referred to as the finite $L_n$-localization and $L_n$-localization, respectively.
\end{defn}

As the terminology suggests, the functors $L_n^f$ are in fact finite localizations, with $\ker(L_n^f)$ generated by any finite type $(n+1)$-spectrum \cite{mahowaldsadofsky, ravenel_liftafter}. It then follows from the thick subcategory theorem that any finite localization functor of the category of spectra which is not equal to the identity or the zero functor must be one of the $L_n^f$. Their key features are summarized in the next proposition, see~\cite{ravenel_liftafter, mahowaldsadofsky, miller_finiteloc}. 

\begin{prop}[Mahowald--Sadofsky, Miller, Ravenel]
For each $n$, the functor $L_n^f$ is a finite and thus smashing localization. If $F$ is a finite type $n$ spectrum then $L_n^fF \simeq \Tel(F)$.
\end{prop}

Having classified all finite localizations, we now turn to the a priori larger set of smashing localizations. The smash product theorem \cite{orangebook} and its proof establish the first part of the next result:

\begin{thm}[Hopkins--Ravenel]
For any $n\ge 0$, the localization functor $L_n$ is smashing.  
\end{thm}

There is a natural transformation $L_n^f \to L_n$ which is an equivalence on all $MU$-module spectra and all $L_i$-local spectra for any $i \ge 0$, as shown in \cite{hovey_csc, hovey_morava_1999}. In other words, there is a close relationship between the functors $L_n$ and their finite counterparts $L_n^f$. As explained in \cite{orangebook}, if the two localizations were in fact equivalent for all $n$, then two naturally arising filtrations on the stable homotopy groups of spheres would coincide, making the computation of $\pi_*S^0$ more amenable to algebraic techniques. This led Ravenel~\cite{ravenel_localization} to:

\begin{conj}[Telescope conjecture]
For any $n \ge 0$, the natural map $L_n^f \to L_n$ is an equivalence. 
\end{conj}

For $n=0$, both $L_0^f$ and $L_0$ identify with rationalization. Based on explicit computations of the homotopy groups of $L_1S^0/p$ and $L_1^fS^0/p = \Tel(S^0/p)$ by Mahowald ($p=2$, \cite{mahowald_bo}) and Miller ($p>2$, \cite{miller_adams}), Bousfield \cite{bousfield_locspectra} deduced: 

\begin{thm}[Bousfield, Mahowald, Miller]
The telescope conjecture holds at height $n=1$. 
\end{thm}

One might thus hope for an inductive approach to the telescope conjecture, passing from height $n-1$ to height $n$. The corresponding relative version admits a number of equivalent formulations, see~\cite{mrs_triple}:

\begin{prop}
Let $n\ge 1$ and suppose $F$ is finite of type $n$, then the following are equivalent:
	\begin{enumerate}
		\item If $L_{n-1}^f \simeq L_{n-1}$, then $L_n^f \simeq L_n$.
		\item $\Tel(F) \simeq L_nF$.
		\item $\langle \Tel(F) \rangle = \langle K(n) \rangle$.
		\item The Adams--Novikov spectral sequence for $\Tel(F)$ converges to $\pi_*\Tel(F)$.
	\end{enumerate}
\end{prop}

Note that, by the thick subcategory theorem, a single example or counterexample that is finite of type $n$ is enough to settle the passage from height $n-1$ to height $n$. 

For $n= 2$ and $p\ge 5$, Ravenel~\cite{ravenel_progress} began the analogue of Miller's height 1 calculation for $V(1)=S^0/(p,v_1)$, attempting to show that the telescope conjecture is false in these cases, but this computation has not yet been completed due to its considerable complexity. In \cite{mrs_triple}, Mahowald, Ravenel, and Shick describe an alternative approach based on a spectrum $Y(n)$ such that $\pi_*L_nY(n)$ is finitely generated over $R(n)_* = K(n)_*[v_{n+1},\ldots,v_{2n}]$, but $\pi_*L_n^fY(n)$ can only be finitely generated over $R(n)_*$ if there is a ``bizarre pattern of differentials'' in the corresponding localized Adams spectral sequence. Thus, if these patterns could be ruled out, we would disprove the telescope conjecture at heights $n\ge 2$. At this time, the telescope conjecture is still open for all $n\ge 2$ and all $p$, and generally believed to be false.

\section{Classification of smashing Bousfield localizations}

This section discusses the classification of smashing Bousfield localizations of the ($p$-local) stable homotopy category. In particular, we prove that the telescope conjecture for all heights $n$ is equivalent to the so-called generalized telescope conjecture (or generalized smashing conjecture). Since this material is more technical than the rest of this survey, the reader may want to skip ahead to the conclusion at the end of this section. We start with two lemmas, the first of which is reminiscent of the type classification of finite spectra. 

\begin{lem}\label{lem:1}
Let $L$ be a smashing localization functor. If $LK(n) \not\simeq 0$, then $LK(n-1) \not\simeq 0$. 
\end{lem}
\begin{proof}
Suppose $LK(n) \not\simeq 0$. Since $K(n) \sm LS^0 \simeq LK(n)$ is a module over $K(n)$ and hence splits into a wedge of shifted copies of $K(n)$, we see that $K(n)$ is $L$-local and thus the canonical map $K(n) \to LK(n)$ is an equivalence. This implies that  $\langle LS^0 \rangle \ge \langle K(n) \rangle$: Indeed, if $X \sm LS^0 \simeq 0$, then $0 \simeq X \sm LS^0 \sm K(n) \simeq X \sm K(n)$ as well. 

The next claim is that $\langle LS^0 \rangle \ge \langle \bigvee_{i=0}^{n}K(i) \rangle$. To this end, note that $L_{K(n)}S^0$ is $K(n)$-local, hence $LS^0$-local. Because $L$ is smashing, we get an equality $\langle LS^0 \sm L_{K(n)}S^0 \rangle = \langle L_{K(n)}S^0 \rangle$, which then yields
\begin{equation}\label{eq:bc}
\langle LS^0 \rangle \ge \langle LS^0 \sm L_{K(n)}S^0 \rangle = \langle L_{K(n)}S^0 \rangle = \langle \textstyle{\bigvee_{i=0}^{n}}K(i) \rangle,
\end{equation}
where the last equality is \cite[Cor.~2.4]{hovey_csc}. Therefore, we have $LK(n-1) \not\simeq 0$.
\end{proof}

The proof of the next lemma requires the nilpotence theorem.

\begin{lem}\label{lem:2}
Suppose $L$ is a smashing localization and $n\ge 0$, then $LK(n) \simeq 0$ if and only if any finite spectrum of type at least $n$ is in $\ker(L)$.
\end{lem}
\begin{proof}
Suppose $LK(n) \simeq 0$ and let $F$ be a finite spectrum of type at least $n$. Replacing $F$ with $\End(F) \simeq DF \sm F$ if necessary, we may assume that $F$ and thus $LF$ are ring spectra. By the nilpotence theorem, it thus suffices to show that $K(i) \sm LF \simeq 0$ for all $0 \le i \le \infty$. Since $L$ is smashing, $K(i) \sm LF \simeq LK(i) \sm F \simeq 0$ for $n \le i \le \infty$ using the assumption and \Cref{lem:1}, while the hypothesis on $F$ guarantees that it also vanishes for $0 \le i < n$.

Conversely, let $F$ be a finite type $n$ spectrum so that $LF \simeq 0$. It follows that $F \sm LK(n) \simeq 0$. But $K(n)$ is a retract of $LK(n)$ provided $LK(n) \not\simeq 0$, so $F \sm K(n) \simeq 0$ as well, contradicting the assumption on $F$. Therefore, $LK(n) \simeq 0$. 
\end{proof}

As the next proof shows, we can use \Cref{lem:1} to detect smashing localizations. 

\begin{prop}
If $L$ is a smashing localization which is neither $0$ nor the identity functor, then there exists an $n\ge 0$ such that $\ker(L_n^f) \subseteq \ker(L) \subseteq \ker(L_n)$.
\end{prop}
\begin{proof}
By \Cref{lem:1}, any smashing localization functor $L$ belongs to one of the following three classes:
	\begin{enumerate}
		\item $LK(n) = 0$ for all $n$, or
		\item there exists an $n$ such that $LK(n) \not\simeq 0$ and $LK(m) \simeq 0$ for all $m > n$, or
		\item $LK(n) \simeq K(n)$ for all $n$.
	\end{enumerate}
In Case (1), $\ker(L)$ contains the sphere spectrum $S^0$ by \Cref{lem:2}, so $L \simeq 0$. If $L$ belongs to the second class, then \Cref{lem:2} shows that $\ker(L_n^f) \subseteq \ker(L)$, so it remains to show that $\ker(L) \subseteq \ker(L_n)$. To this end, let $X \in \ker(L)$. Because $L$ is smashing, this implies $LS^0 \sm X \simeq 0$ and thus $\bigvee_{i=0}^{n}K(i) \sm X \simeq 0$ by \eqref{eq:bc}. Therefore, $X \in \ker(L_n)$ as desired. 

Finally, if $LK(n) \simeq K(n)$ for all $n$, then \eqref{eq:bc} implies that any $L_n$-local spectrum is $L$-local, so $S^0 \simeq \lim_nL_nS^0$ is $L$-local by the chromatic convergence theorem. Therefore, $L$ must be equivalent to the identity functor, again using that $L$ is smashing.
\end{proof}

\begin{cor}
The telescope conjecture holds for all $n$ if and only if all smashing localization functors on the stable homotopy category are finite. 
\end{cor}

This latter formulation, originally due to Bousfield \cite[Conj.~3.4]{bousfield_locspectra}, generalizes well to other compactly generated triangulated categories where it has been studied extensively, see for example \cite{keller_smashing,krause_telescopeconj}.

\section{The chromatic splitting conjecture}

The chromatic splitting conjecture describes how the localizations $L_nS^0$ for varying $n$ assemble into $S^0$ via the chromatic tower \eqref{eq:chromatictower}, working $p$-locally as before. Informally speaking, it asserts that this gluing process is as simple as it can be without being trivial, but there are various refinements of its statement. We will focus on the weakest form here and refer the interested reader to \cite{hovey_csc} for further details. 

For each $n\ge 1$ there is a map of fiber sequences, where the right square---known as the chromatic fracture square---is a homotopy pullback:
\begin{equation}\label{eq:chromaticsquare}
\xymatrix{F(L_{n-1}S^0,L_nX) \ar[r] \ar[d]_{\simeq} & L_nX \ar[r] \ar[d] & L_{K(n)}X \ar[d] \ar@{-->}[ld]_-{\alpha_n} \\ 
F(L_{n-1}S^0,L_nX) \ar[r]_-{\beta_n} & L_{n-1}X \ar[r] & L_{n-1}L_{K(n)}X. \ar@<0.5ex>@{-->}[l]^-{\gamma_n}}  
\end{equation}

Consider the question whether there exists a map $\alpha_n$ as indicated making the top triangle in the chromatic fracture square commute. By chasing the diagram, such a map exists if and only if $\beta_n$ is nullhomotopic, which in turn is equivalent to the existence of a map $\gamma_n$ splitting the map $L_{n-1}X \to L_{n-1}L_{K(n)}X$. Based on explicit computations of the cohomology of Morava stabilizer groups as well as of $\pi_*L_{K(n)}S^0$ for small $n$, Hopkins (see~\cite{hovey_csc}) arrived at the following: 

\begin{conj}[Chromatic splitting conjecture]\label{conj:csc}
If $X$ is the $p$-completion of a finite spectrum, then a splitting $\gamma_n$ exists for all $n$.
\end{conj}

The finiteness assumption on $X$ is essential in this conjecture: Indeed, Devinatz \cite{devinatz_counterexample} proves that, for $X = BP_p$ the $p$-completion of the Brown--Peterson spectrum, the map $L_{n-1}BP_p \to L_{n-1}L_{K(n)}BP_p$ splits if and only if $n = 1$. If the chromatic splitting conjecture holds for a finite spectrum $X$, then we obtain the following consequences:
\begin{enumerate}
	\item The canonical map $X_p \to \prod_n L_{K(n)}X_p$ is the inclusion of a summand, as proven in \cite{hovey_csc}.
	\item Taking the limit over the compositions $L_{K(n+1)}X \xrightarrow{\alpha_{n+1}} L_nX \to L_{K(n)}X$ gives an equivalence $X \to \lim_nL_{K(n)}X$. This follows follows from the chromatic convergence theorem by cofinality.
\end{enumerate}
In other words, the chromatic splitting conjecture implies that a finite spectrum $X$ can be recovered from its monochromatic pieces $L_{K(n)}X$. 

We now review what is known about the chromatic splitting conjecture for $S_p^0$, the $p$-complete sphere spectrum. Take $n=1$ and $p>2$, then a classical computation with complex $K$-theory, originally due to Adams and Baird~\cite{adamsbaird} and then revisited by Ravenel~\cite{ravenel_localization}, shows that 
\[
\pi_iL_{K(1)}S_p^0 \cong
	\begin{cases}
		\Z_p & \text{for } i \in \{-1,0\}, \\
		\Z/p^{s+1} & \text{for } i = 2(p-1)p^sm-1 \text{ with } p \nmid m, \\
		0 & \text{otherwise}.
	\end{cases}
\]
Thus $\pi_*L_0L_{K(1)}S^0_p \cong \Q_p$ for $i = 0$ and $i= -1$ and is $0$ otherwise; of course, $\pi_*L_0S_p^0$ is isomorphic to $\Q_p$ in degree $0$. One can then see that $L_0L_{K(1)}S_p^0$ splits as $L_0S_p^0 \vee L_0S_p^{-1}$. Replacing complex $K$-theory by real $K$-theory yields the same conclusion for $n=1$ and $p=2$. The analogous computations at height $n=2$ are considerably more complex and are the subject of extensive work by Shimomura--Yabe \cite{sy_25} ($p \ge 5$), Goerss--Henn--Mahowald--Rezk \cite{ghmr_23,ghm_csc23} ($p=3$), and Beaudry--Goerss--Henn \cite{bgh_csc22} ($p=2$). Their results can be summarized as follows:

\begin{thm}[Beaudry--Goerss--Henn--Mahwald--Rezk--Shimomura--Yabe]
The chromatic splitting conjecture holds for $n=2$ and all $p$. If $p\ge 3$, then 
\[
L_1L_{K(2)}S_p^0 \simeq L_1(S_p^0 \vee S_p^{-1}) \vee L_0(S_p^{-3} \vee S_p^{-4}),
\]
while for $p=2$, we have
\[
L_1L_{K(2)}S_p^0 \simeq L_1(S_p^0 \vee S_p^{-1} \vee S_p^{-2}/p \vee S_p^{-3}/p) \vee L_0(S_p^{-3} \vee S_p^{-4}).
\]
\end{thm}

There is a stronger version of \Cref{conj:csc} which additionally describes how the fiber term $F(L_{n-1}S^0,L_nX)$ in \eqref{eq:chromaticsquare} decomposes into spectra of the form $L_iX$ with $0 \le i \le n-1$. If correct, it would imply (see~\cite{acsc}) that the stable homotopy groups of $L_{K(n)}S^0$ are finitely generated over $\Z_p$ for $n\ge 1$, another major open problem in chromatic homotopy theory, see \cite{devinatz_finiteness} for partial results. However, this conjecture is open for all heights $n \ge 3$ and primes $p$; there are hints \cite{ultra1,pstragowski_alg} that the problem might at least be approachable for large primes with respect to the height $n$. 

We end this section with the following result by Minami \cite{minami_chromatictower}, which provides some evidence for the chromatic splitting conjecture at general heights. He introduces a class of so-called robust spectra including finite spectra as well as $BP$ and proves:

\begin{thm}[Minami]
Fix a height $n$ and prime $p$. If $X$ is a robust spectrum and $m$ and $k$ are positive integers satisfying $m-k \ge n+ s_0 +1$ where $s_0$ is the vanishing line intercept of the $E(n)$-based Adams--Novikov spectral sequence for $S^0$, then the map $L_mX \to L_nX$ factors through $L_{K(k+1) \vee \ldots \vee K(m)}X$. 
\end{thm}

\section{An algebraic analogue}

We conclude this survey by discussing an algebraic analogue of the stable homotopy category in which algebraic versions of the generating hypothesis, the telescope conjecture, as well as the chromatic splitting conjecture have been settled. This is just one instance of the observation that the chromatic programme and consequently the above chromatic conjectures can be formulated in many other contexts, thereby providing a plethora of test cases as well as motivation for a fruitful transfer of techniques. Other examples include derived categories of quasi-coherent sheaves on schemes or stacks, stable equivariant homotopy categories, motivic categories, or categories arising in non-commutative geometry, see \cite{balmer_icm} for an overview. 

Let $G$ be a finite group, let $k$ be a field of characteristic $p$, and write $kG$ for the associated group algebra. Recall that the stable module category $\StMod_{kG}$ is the quotient of $\Mod_{kG}$ by the projectives and that it comes equipped with the structure of a symmetric monoidal triangulated category with tensor unit $k$. As in~\cite{bik_finitegroups} we write $\Proj(H^*(G;k))$ for the projective variety of the Noetherian graded commutative ring $H^*(G;k)$; the underlying set of $\Proj(H^*(G;k))$ consists of the homogeneous prime ideals in $H^*(G;k)$ different from the ideal of all positive degree elements.

The finite localization functors on $\StMod_{kG}$ have been classified in the work of Benson, Carlson, and Rickard~\cite{bcr_thick}. As a result of a series of papers culminating in \cite{bik_finitegroups}, Benson, Iyengar, and Krause generalized this to a complete classification of all localization functors: They develop a theory of support and employ it to establish a bijection between the set of localizing tensor ideals of $\StMod_{kG}$ and arbitrary subsets of $\Proj(H^*(G;k))$. Their theory yields in particular a proof of the telescope conjecture in this context.

\begin{thm}[Benson--Iyengar--Krause]
The generalized telescope conjecture holds in $\StMod_{kG}$, i.e., the category of acyclics of any smashing localization functor is generated by compact objects. Furthermore, the smashing localization functors on $\StMod_{kG}$ are in bijection with specialization closed\footnote{A subset $\cV$ is called specialization closed if $\fp \in \cV$ and $\fp \subseteq \fq$ imply $\fq \in \cV$.} subsets of $\Proj(H^*(G;k))$. 
\end{thm}

In fact, they establish an analogous classification for the larger category $\Stable_{kG}$ of unbounded complexes of injective $kG$-modules up to homotopy, which fits into a recollement between $\StMod_{kG}$ and the derived category of $kG$-modules~\cite{krausebenson_kg}. In this case, the role of the parametrizing variety is played by $\Spec^h(H^*(G;k))$, the Zariski spectrum of all homogeneous prime ideals of $H^*(G;k)$. In particular, any specialization closed subset $\cV \subseteq \Spec^h(H^*(G;k))$ gives rise to a localization functor $L_{\cV}$ on $\Stable_{kG}$. For example, if $\fp$ is a homogeneous prime ideal, then $\cV(\fp) = \{\fq\mid \fp \subseteq \fq \} \subseteq \Spec^h(H^*(G;k))$ is specialization closed, and thus provides a localization functor $L_{\cV(\fp)}$ and a completion functor $\Lambda^{\fp}$. These functors should be thought of as algebraic analogues of the functor $L_{n-1}$ and $L_{K(n)}$. 

Before we can state the analogue of the chromatic splitting conjecture in this context, we need to introduce some terminology: To emphasize the analogy to stable homotopy category, we write $\pi_*M$ for the graded abelian group of homotopy classes of maps from $k$ to $M$ in $\Stable_{kG}$. Call two prime ideals $\fp,\fp' \in \Spec^h(H^*(G;k))$ adjacent if $\fp' \subsetneq \fp$ and this chain does not refine, i.e., there does not exist $\fq \in \Spec^h(H^*(G;k))$ such that $\fp' \subsetneq \fq \subsetneq \fp$. Furthermore, a module $M \in \Stable_{kG}$ is said to be $\fp$-local if $\pi_*M$ is a $\fp$-local $H^*(G;k)$-module, and a compact $M$ is said to be of type $\fp'$ if $\pi_*M$ is $\fp'$-torsion as a graded $H^*(G;k)$-module. 

\begin{thm}[\cite{bhv3}]
Suppose $G$ is a finite $p$-group. Let $\fp,\fp' \in \Spec^h(H^*(G;k))$ be adjacent prime ideals and let $M \in \Stable_{kG}$ be $\fp$-local. There is a homotopy pullback square
\[
\xymatrix{M \ar[r] \ar[d] & \Lambda^{\fp}M \ar[d] \\
L_{\cV(\fp)}M \ar[r] & L_{\cV(\fp)}\Lambda^{\fp}M.}
\]
If $M$ is compact and of type $\fp'$, then the bottom map in this square is split. 
\end{thm}

Finally, we consider the analogue of the generating hypothesis in $\StMod_{kG}$, which asserts that a map $f\colon M \to N$ between finitely generated modules is nullhomotopic, i.e., factors through a projective module, if and only if $\pi_*f = 0$. Based on earlier work of \cite{bccm_genhyp} in the $p$-group case, \cite{ccm_genhyp} gives a complete answer:

\begin{thm}[Benson--Carlson--Chebolu--Christensen--Min\'{a}\v{c}]
The generating hypothesis holds for $\StMod_{kG}$ if and only if the $p$-Sylow subgroup of $G$ is isomorphic to either $C_2$ or $C_3$. 
\end{thm}

The techniques used in their proof, namely Auslander--Reiten theory, carry over to the chromatic setting to establish the failure of a $K(n)$-local analogue of the generating hypothesis \cite{barthel_ars}, thereby bringing us back to our starting point.

\biblio
\bibliography{bibliography}\bibliographystyle{plain}
\end{document}